\documentclass[11pt]{amsart}

 \usepackage{amsmath,amssymb}
 \usepackage{color}
 \usepackage{mathrsfs}
 \usepackage{graphicx}
 \usepackage{lineno}
 \usepackage[colorlinks, pagebackref, linkcolor=red, citecolor=blue]{hyperref}

 \newcommand{\zit}[1]{(\ref{#1})}

 \numberwithin{equation}{section}
 \theoremstyle{plain}
 \newtheorem{theorem}{Theorem}[section]
 \newtheorem{lemma}[theorem]{Lemma}

 \newtheorem{proposition}[theorem]{Proposition}
 
   \theoremstyle{definition}
  \newtheorem{definition}[theorem]{Definition}
   \newtheorem{example}[theorem]{Example}
   \newtheorem{remark}[theorem]{Remark}

 \def\bsr{\operatorname{bsr}}
\def\tsr{\operatorname{tsr}}

 \def\R{ \mathbb R}

 \def\C{{ \mathbb C}}
 \def\Z{{ \mathbb Z}}
 \def\N{{ \mathbb N}}
 \def\B{{ \mathbf B}}

 \def\e{\varepsilon}
 
 \def\sp {\quad}
 \def\ssc{\scriptscriptstyle}
 \def\sc{\scriptstyle}
 
 \def\bs{\boldsymbol}
 \def\dis{\displaystyle}
 \def\union{\cup}
 \def\Union{\bigcup}
 \def\inter{\cap}
 \def\Inter{\bigcap }
 \def\ov{\overline}

 \def\ss{\subseteq}
 \def\emp{\emptyset}

 \def\buildrel#1_#2^#3{\mathrel{\mathop{\kern 0pt#1}\limits_{#2}^{#3}}}
 
 \def\imp{\Longrightarrow}

\begin{document}

 \title [The polynomial ring]
 {The ring of real-valued multivariate polynomials:
   an analyst's perspective}

 %\thanks{}
 
 \author{Raymond Mortini}
  \address{
   \small Universit\'{e} de Lorraine\\
\small D\'{e}partement de Math\'{e}matiques et  
Institut \'Elie Cartan de Lorraine,  UMR 7502\\
\small Ile du Saulcy\\
 \small F-57045 Metz, France} 
 \email{raymond.mortini@univ-lorraine.fr}

\author{Rudolf Rupp}
\address{ Fakult\"at Allgemeinwissenschaften\\
\small Georg-Simon-Ohm-Hochschule N\"urnberg\\
\small Kesslerplatz 12\\
\small D-90489 N\"urnberg, Germany
}
\email  {Rudolf.Rupp@ohm-hochschule.de}


 
 \begin{abstract}
 In this survey we determine an explicit set of generators of the maximal ideals
 in the ring $\R[x_1,\dots,x_n]$ of polynomials in $n$ variables with real coefficients
 and give an easy analytic proof of the Bass-Vasershtein theorem on the Bass stable
 rank of $\R[x_1,\dots,x_n]$. The ingredients of the proof stem from different publications
 by Coquand, Lombardi, Estes and Ohm. We conclude with a calculation of the topological stable
 rank of  $\R[x_1,\dots,x_n]$, which seems to be unknown so far.
  \end{abstract}

  \maketitle

 %\centerline {\small\the\day.\the \month.\the\year} \medskip
 \section*{Introduction}
 
 In his seminal paper \cite[Theorem 8]{va},  that paved the way to all future
 investigations of the Bass stable rank for function algebras,
 L. Vasershtein  deduced from a Theorem of  H. Bass \cite{ba}
 (see below)  that  the Bass stable rank of the ring of real polynomials in $n$ variables 
 is $n+1$. Since for an analyst Bass' fundamental paper is very hard to understand
 it is desirable to develop an analytic  proof of Bass' important result that  can  easily be read.  
 This was done in a paper by Estes and Ohm \cite{eo}. The whole depends on the determination
 of the Krull dimension of $\R[x_1,\dots,x_n]$, the  known proofs prior to 2005  were rather
 involved. But
 also here, a nice elementary proof had been developed around 2005 by
  Coquand and Lombardi \cite{colo}.
 Their  short proof depends on  the standard algebraic tool  of ``localization of rings".
 We shall replace this by a direct construction of a chain of prime ideals of length $n$
and obtain in this way   an entirely analytic proof of the Bass-Vasershtein Theorem.
The only tool used in the proof will now be Zorn's Lemma. 
In our survey we present all these proofs so that  it will be entirely self-contained;
it will no longer be necessary to look up half a dozen papers in order to admire
this nice result by Bass and Vasershtein. We conclude the paper with a  result we could not trace
in the literature: the
determination of the topological stable rank of $\R[x_1,\dots,x_n]$: every $(n+1)$-tuple of 
real-valued polynomials can be uniformly approximated on $\R^n$ be invertible $(n+1)$-tuples
in $\R[x_1,\dots,x_n]$.

This survey forms part of an ongoing textbook project on stable ranks of function algebras,
due to be finished only in a couple of years from now (now = 2013).  Therefore we decided
to  make this chapter already  available to the mathematical community (mainly for readers
of this Proceedings and for master students interested  in function theory and function algebras).
 
 \section{The maximal ideals of $\R[x_1,\dots,x_n]$}
 Associated with   $\R[x_1,\dots,x_n]$ is the following algebra of {\it real-symmetric polynomials}:
$$\C_{{\rm sym}}[z_1,\dots, z_n]=\left\{f\in \C[z_1,\dots, z_n]: f(z_1,\dots, z_n)= 
\ov{f(\ov z_1,\dots, \ov z_n)}\; \;\forall (z_1,\dots,z_n)\in \C^n\;\right\}.$$
For shortness we write $\bs z$ for the  $n$-tuples $(z_1,\dots, z_n)$ and $\ov {\bs z}$
for $(\ov z_1,\dots, \ov z_n)$.

\begin{lemma}\label{symvsreal}
$\C_{{\rm sym}}[z_1,\dots, z_n]$ is a real algebra of complex-valued polynomials 
 that is real-isomorphic to $\R[x_1,\dots, x_n]$.
\end{lemma}
\begin{proof}
It is easy to see that $\C_{{\rm sym}}[\bs z]$ is a real algebra.
Let $\rho: \C_{{\rm sym}}[\bs z]\to \R[\bs x]$ be the restriction map $p\mapsto p|_{\R^n}$.
Note that $\rho$ is well defined, since the coefficients of a polynomial in $\C_{{\rm sym}}[\bs z]$
are real; in fact if
$$p(\bs z)=\sum_{\bs n\in I}  a_{\bs n} \mathbf z^{\bs n}\in \C_{{\rm sym}}[\bs z],$$
then 
$$
\ov{p(\ov{\bs z})}= \sum_{\bs n\in I} \ov{ a_{\bs n}} \mathbf z^{\bs n}
= \sum_{\bs n\in I}  a_{\bs n} \mathbf z^{\bs n}.
$$
The uniqueness of the coefficients implies that $\ov{ a_{\bs n}}= a_{\bs n}$. Hence
$ a_{\bs n}\in \R$.
The rest is clear.
\end{proof}

For $\bs a\in  \C^n$ let 
$$M_{\bs a}:=\{p\in  \C[z_1,\dots,z_n]: p(\bs a)=0\}.$$
By Hilbert's Nullstellensatz (see for instance \cite{lang} or \cite{kap})
an ideal in $\C[z_1,\dots,z_n]$ is maximal if and only if
it has the form $M_{\bs a}$ for some $\bs a\in \C^n$. 
This will be used  in the following to determine the class of maximal ideals in $\C_{{\rm sym}}[z_1,\dots, z_n]$.
 
 \begin{theorem}\label{maxidi}%\label{maxisympol1}
The class of maximal ideals of  $\C_{{\rm sym}}[z_1,\dots, z_n]$ coincides with the class
of ideals of the form 
$$S_{\bs a}:=M_{\bs a}\inter M_{\ov{\bs a}}\inter \C_{{\rm sym}}[z_1,\dots, z_n],$$
where $\bs a\in \C^n$.
The set  $\{\bs a, \ov{\bs a}\}$ is uniquely determined for a given maximal ideal.
\end{theorem}
\begin{proof}
We first note that $M_{\bs a}\inter M_{\ov{\bs a}}\inter \C_{{\rm sym}}[\bs z]=
M_{\bs a}\inter \C_{{\rm sym}}[\bs z]$, because for every polynomial $p$  in 
$ \C_{{\rm sym}}[\bs z]$ it holds that $p(\bs a)=0$  if and only if $p(\ov{\bs a})=0$.

Next we show that the ideals $S_{\bs a}$ are maximal.   So 
suppose that $f\in \C_{{\rm sym}}[\bs z]$ does not vanish at $\bs a$. Then 
  $$\bigl(f-f(\bs a)\bigr)\;\bigl(f-\ov {f(\bs a)}\bigr)= f^2-\bigl(2{\rm Re}\; f(\bs a)\bigr)\, f +|f(\bs a)|^2
  \in S_{\bs a}$$ 
  and 
 $$
  1=\frac{\bigl(f-f(\bs a)\bigr)\;\bigl(f-\ov {f(\bs a)}\bigr)}{|f(\bs a)|^2}- f \;
  \frac{ f-(f(\bs a)+\ov {f(\bs a)})}{|f(\bs a)|^2}.
  $$
  Hence the ideal, $I_{\C_{{\rm sym}}[\bs z]}(S_{\bs a},f)$, generated by $S_{\bs a}$ 
  and $f$ is the whole algebra and so $S_{\bs a}$
   is maximal.
  We note that in the case where $f(\bs a)$ is real, we simply could argue as follows, since the
  constant functions $\bs z\mapsto f(\bs a)$ and $\bs z\mapsto 1/f(\bs a)$ then belong to $\C_{{\rm sym}}[\bs z]$:
  $$1=-\frac{f-f(\bs a)}{f(\bs a)} + \frac{f}{f(\bs a)} \in I_{\C_{{\rm sym}}[\bs z]}(S_{\bs a},f).$$
  
  It remains to show that every maximal ideal $M$  in $\C_{{\rm sym}}[\bs z]$  coincides with
  $S_{\bs a}$ for some $\bs a\in \C^n$.
  Suppose, to the contrary, that $M$ is not contained in any ideal of the form $S_{\bs a}$.
  Hence, for every $\bs a\in \C^n$, there is $p_{\bs a}\in M$ such that  $p_{\bs a}(\bs a)\not=0$.
By Hilbert's Nullstellensatz,
 the ideal generated by  the set $S=\{p_{\bs a}: \bs a\in \C^n\}$
in $\C[\bs z]$ coincides with $\C[\bs z]$.  Hence there are $q_j\in \C[\bs z]$ and finitely
many $\bs a_j\in \C^n$,  $(j=1,\dots, N)$,  such that
$$ \sum_{j=1}^N q_j \,p_{\bs a_j}=1.$$
Now, by taking complex conjugates, and using the fact that  
$p_{\bs a_j}\in M\ss \C_{{\rm sym}}[\bs z]$,
we get
$$1=\ov{\sum_{j=1}^N q_j(\ov{\bs z}) \, p_{\bs a_j}(\ov{\bs z})}= 
 \sum_{j=1} \ov{q_j(\ov{\bs z})}\,  p_{\bs a_j}(\bs z).$$
 Hence, with $$q_j^*(\bs z) =\frac{1}{2}\bigl(\ov{q_j(\ov{\bs z})} + q_j(\bs z)\bigr),$$
 we conclude that
 $$\sum_{j=1}^N q_j^* \, p_{\bs a_j}=1.$$
 Since $q_j^*\in \C_{{\rm sym}}[\bs z]$ and $p_{\bs a_j}\in M$ 
 we obtain the contradiction that $1\in M$. Thus $M\ss S_{\bs a}$ for some $\bs a\in\C^n$. 
 The maximality of $M$ now implies that $M=S_{\bs a}$.
 
Finally we show the uniqueness of  $\{\bs a,\ov{\bs a}\}$. So suppose that 
$\bs b\not\in \{\bs a, \ov{\bs a}\}$. 

{\sl Case 1} There is an index $i_0$ such that $b_{i_0}\notin\{a_{i_0}, \ov{a_{i_0}}\}$.
Then the polynomial $p$, given by
$$p(z_1,\dots, z_n)=(z_{i_0}-a_{i_0})(z_{i_0}-\ov{a_{i_0}})$$
belongs to $ \C_{{\rm sym}}[z_1,\dots,z_n]$,
vanishes at $\bs a$, but not at $\bs b$
 (if $a_{i_0}\in \R$, then it suffices to take
$p(z_1,\dots, z_n)=z_{i_0}-a_{i_0}$).

{\sl Case 2} There are two indices   $i_0$  and $i_1$ such that  $a_{i_\nu}\notin \R$, 
$b_{i_0}=\ov {a_{i_0}}$ and  $b_{i_1}=a_{i_1}$.
Then the polynomial $q$ given by
$$q(z_1,\dots,z_n)= \Bigl((z_{i_{0}}-a_{i_0})+ (z_{i_1}-a_{i_1})\Bigr) \;\cdot\;
 \Bigl((z_{i_{0}}-\ov{a_{i_0}})+ (z_{i_1}-\ov{a_{i_1}})\Bigr)
$$ 
belongs to $ \C_{{\rm sym}}[z_1,\dots,z_n]$, vanishes at $\bs a$, but not at $\bs b$.

Hence, in both cases,  $p\in S_{\bs a}\setminus S_{\bs b}$.
There are no other cases left. 
\end{proof}

\begin{theorem}\label{bezpolreal}
The B\'ezout equation $\sum_{j=1}^k q_jp_j=1$ admits a solution in the ring
$R=\C_{{\rm sym}}[z_1,\dots,z_n]$ or $\R[x_1,\dots,x_n]$
if and only if the polynomials $p_j$ do not have a common zero in $\C^n$.
\end{theorem}
\begin{proof}
By the identity  theorem for holomorphic functions of several complex variables, the condition $\sum_{j=1}^k q_jp_j=1$ on $\R^n$ implies that the same equality holds
on $\C^n$. Hence the given polynomials $p_j$ do not have a common zero in $\C^n$.\\
Conversely, if the $p_j$ do not have a common zero in $\C^n$ then, by Theorem
\ref{maxidi}, the ideal generated by  the 
 $p_j$ in $R$ cannot be a proper ideal.
Hence there are $q_j\in R$ such that $\sum_{j=1}^k q_jp_j=1$.
\end{proof}

%%%%%
We shall now determine an explicit class of generators for the maximal ideals
in $\R[x_1,\dots,x_n]$. Recall that
by Hilbert's Nullstellensatz  the maximal ideals in $\C[z_1,\dots,z_n]$ are generated by $n$ polynomials of the form $z_1-a_1,\dots, z_n-a_n$, where 
$\bs a:=(a_1,\dots, a_n)\in \C^n$.
The situation for the real algebra $\R[x_1,\dots,x_n]$ is quite different. Here are some examples that will reflect the general situation dealt with below. We identify $\R[x_1,\dots,x_n]$ with
$\C_{{\rm sym}}[z_1,\dots,z_n]$. The proof of the assertions is left as an exercise to the reader.
%%%%
\begin{example}\label{csym}\hfill
\begin{enumerate}
\item [(1)]  Let $\sigma\in \C\setminus \R$ and $r_j\in \R$, $j=1,2, \dots, n-1$.
Then the ideal generated  by $$x_n^2-(2\,{\rm Re}\; \sigma)\;x_n +|\sigma|^2$$ and
$x_j-r_j$,   $ (j=1,\dots, n-1)$, is maximal in  $\R[x_1,\dots, x_n]$.
It corresponds to the ideal  $S_{(r_1,\dots,r_{n-1},\sigma)}$.
\item [(2)] The ideal $I_{\R[x,y]}( 1+x^2, 1+y^2)$ generated by $1+x^2$ and $1+y^2$
is not maximal.
\item[(4)]  The ideal $M:=I_{\R[x,y]}( 1+x^2, 1+y^2, 1+xy, x-y)$ is maximal
and corresponds  to  $S_{(i,i)}$.
\item[(5)]  The following representations hold:
\begin{eqnarray*}
M&=&I_{\R[x,y]}( 1+x^2, 1+y^2, x-y)\\
&=& I_{\R[x,y]}( 1+x^2, 1+y^2, 1+xy)\\
&=& I_{\R[x,y]}( 1+xy, x-y).
\end{eqnarray*}
\end{enumerate}
\end{example}

 \begin{theorem}\label{maxireal}
Modulo a re-enumeration of the indices, the maximal ideals 
$M$ of $R:=\R[x_1,\dots,x_n]$ are generated by polynomials  of the form
$$p_j:=x_j-r_j, \; (j=1,\dots, k),$$
$$p_{k+j}:= x_{k+j}^2 -(2{\rm Re}\; a_{j})\;x_{k+j} +|a_{j}|^2, \; (j=1,\dots, m)$$
 {\rm (} $r_j \in\R$, $a_j\in \C\setminus \R$,  $k+m=n$ {\rm )},
and $2^{n-k}-2$ multilinear polynomials $q_j$ in $\R[x_{k+1},\dots, x_n]$
vanshing at $a_{k+1}, \dots, a_n$.
More precisely, we have
$$M= \sum_{j=1}^n p_j(x_j) R +  \sum_{j=1}^{2^{n-k}-2} q_j(x_{k+1},\dots, x_n)\;\R.$$
\end{theorem}
\begin{proof}
Since $\R[x_1,\dots,x_n]$ is isomorphic to  $S:=\C_{{\rm sym}}[z_1, \dots, z_n]$,
it suffices to show that every maximal ideal $S_{\bs a}$ in  $S$ is generated 
by polynomials of the desired type (Theorem \ref{maxidi}).
Fix $\bs a\in \C^n$.  We may assume that 
$$\bs a=(r_1,\dots, r_k, a_{k+1}, \dots, a_{k+m}),$$
with $k+m=n$, where $r_j\in \R$  and $a_{k+1}, \dots, a_{k+m}\in \C\setminus \R$.
Note that $k$ or $m$ may be $0$.  Let $f\in S_{\bs a}$ and
$\bs z=(z_1,\dots, z_n)$.
By the Euclidean division procedure
$$f(\bs z)=\sum_{j=1}^n p_j(\bs z) q_j (\bs z)+ r(z_{k+1},\dots, z_{k+m}),$$ 
where $\deg_{z_j} r< 2$ for $k+1\leq j \leq k+m=n$. Hence $r$ is a multilinear polynomial
of the form 
$$r(z_{k+1},\dots,z_{k+m})=\sum_{\bs j} c_{\bs j} z_{k+1}^{j_1}\cdots z_{k+m}^{j_m},$$
$\bs j=(j_1,\dots,j_m)$, $j_\ell \in \{0,1\}$, $ c_{\bs j}\in \R$.
Moreover, $r(a_{k+1},\dots, a_{k+m})= f(\bs a)=0$.
Now the real vector-space, $V$,
of all  multilinear real-symmetric polynomials  in $m$ variables
has the algebraic dimension $2^{m}$.  Hence, the subspace $V^*$ of all $p\in V$ with 
$p(a_{k+1},\dots, a_{k+m})=p(\ov{a_{k+1}},\dots, \ov{a_{k+m}})=0$ has  dimension $2^{m}-2$.  
Let $\{q_1,\dots, q_{2^m-2}\}$ be
a basis of $V^*$.  Then
$$f\in \sum_{j=1}^n (p_j \,R) + \sum_{j=1}^{2^m-2 } (q_j\, \R).$$
\end{proof}

 We shall now unveil an explicit basis for $V^*$ whenever  $a_j=i$ for  every $j$.
 
 \begin{lemma}\label{basis}
 Let $\bs i=(i,\dots,i)\in \C^m$.  Then  a (vector-space) basis of
 $$V^*=\Bigl\{f(z_1,\dots, z_m)=\sum_{j_1,\dots, j_m} c_{\bs j} z_{k+1}^{j_1}\cdots z_{k+m}^{j_m},\;
j_\ell \in \{0,1\},\; c_{\bs j}\in \R,\;  f(\bs i)=0\Bigr\}$$
is given by
\begin{eqnarray*} x_1-x_j & &1<j \leq  m\\
1+x_{j_1}\, x_{j_2}&& 1\leq j_1<j_2\leq m\\
x_1+x_{j_1}\, x_{j_2}\, x_{j_3}& & 1\leq j_1<j_2<j_3\leq m\\
1-\prod_{\ell=1}^4 x_{j_\ell}&& 1\leq j_1<\dots<j_4\leq m\\
&&\\
x_1-\prod_{\ell=1}^5 x_{j_\ell}&& 1\leq j_1<\dots<j_5\leq m\\
\dots\dots&&\dots\dots
\end{eqnarray*}
The last element has exactly one of the following forms:
$$\begin{cases}\vspace{2mm}
 x_1-\prod_{j=1}^m x_j & \text{if $m\equiv 1 \mod 4$}\\\vspace{2mm}
 1\;\;+\prod_{j=1}^m x_j   & \text{if $m\equiv 2 \mod 4$}\\ \vspace{2mm}
  x_1+\prod_{j=1}^m x_j & \text{if $m\equiv 3 \mod 4$}\\
 1\;\;-\prod_{j=1}^m x_j & \text{if $m\equiv 0 \mod 4$}
\end{cases}
$$\end{lemma}

 \begin{proof}
 All the polynomials above vanish at $\bs i$ (this is the reason for 
 their cyclic behaviour $\mod 4$).
 Moreover,  they are linear independent and there are exactly $2^m-2$  of them.
 Note that  the second summand has the form 
 $$ x_1^{\e_1}x_2^{\e_2}\cdots x_m^{\e_m},\sp \e_j\in\{0,1\}, $$
 the monomials $1=\prod_{j=1}^m x_j^0$ and $x_1=x_1\prod_{j=2}^m x_j^0$ being excluded. 
\end{proof} 
We conclude this section with the final form of the generators of the maximal ideals
in $\R[x_1,\dots,x_n]$. 
\begin{theorem}
Let $m+k=n$, $m\geq 2$,  and 
$\bs a:=( i,\dots, i, r_{m+1},\dots, r_{m+k})\in \C^m\times \R^{k}$.
The maximal ideal $S_{\bs a}$ of $\R[x_1,\dots,x_n]$ is generated 
 by the $2^{m}-2$  multilinear polynomials in Lemma \ref{basis} and the polynomials
 $$p_{m+j}:=x_{m+j}-r_{m+j}, \; (j=1,\dots, k).$$
\end{theorem}
  
  \begin{proof}
  Using Theorem \ref{maxireal}, it suffices to show that the quadratic polynomials
  $1+x_j^2$, $j=1, \dots, m$,  belong to  the ideal generated  by the $2^{m}-2$  multilinear polynomials in Lemma \ref{basis}.
   This is clear, however, in view of the following relations:
  $$ \mbox{$1+x_j^2= -(x_1-x_{j}) x_j+ (1+x_1x_{j})$  for $j=2,\dots, m$},$$
  $$ 1+x_1^2= (x_1-x_m) x_1 + (1+x_1x_m).$$
  \end{proof}
  
  The general case of an arbitrary maximal ideal $S_{\bs a}$  is easily deduced by using the transformation
$$\chi(z_1, \dots, z_m)=\left( \dis \frac{z_1-\alpha_1}{\beta_1},\; \dots, \;  
\frac{z_m-\alpha_m}{\beta_m}\right)$$
of $\C^m$ onto $\C^m$, whenever
$$\bs a=(\alpha_1+i\beta_1, \dots, \alpha_m+i\beta_m, r_{m+1},\dots, r_{m+k})\in \C^m\times\R^k
\ss\C^n,$$
with $\beta_j\not=0$ for $j=1,\dots, m$.

Using more algebraic methods, it can be shown, that every maximal ideal in $R:=F[x_1,\dots, x_n]$ is generated by $n$ elements (see \cite[p. 20]{kap}), where $F$ is a field.  Finally, 
let us mention that $R$ is a Noetherian ring (this means that every ideal in $R$ is finitely
generated; this is Hilbert's basis theorem).

 \section{The Bass stable rank of $\R[x_1,\dots,x_n]$}
    
    \begin{definition}
   Let $R$ be  a commutative unital ring  with identity element 1. We assume that
   $1\not=0$, that is $R$ is not the trivial ring $\{0\}$.

\begin{enumerate}

\item[(1)] If $a_j\in R$, $(j=1,\dots, n)$,  then 
$$I_R(a_1,\dots,a_n):=\Bigl\{\sum_{j=1}^n x_j a_j: x_j\in R\Bigr\}$$
is the ideal generated by the $a_j$ in $R$.

\item [(2)] An $n$-tuple $(f_1,\dots,f_n)\in R^n$ is said to be {\it invertible} (or {\it unimodular}), 
  if there exists
 $(x_1,\dots,x_n)\in R^n$ such that the B\'ezout equation $\sum_{j=1}^n x_jf_j=1$
 is satisfied.
   The set of all invertible $n$-tuples is denoted by $U_n(R)$. Note that $U_1(R)=R^{-1}$.
   
 \parindent=0pt An $(n+1)$-tuple $(f_1,\dots,f_n,g)\in U_{n+1}(R)$ is  called {\it reducible}  
if there exists 
 $(a_1,\dots,a_n)\in R^n$ such that $(f_1+a_1g,\dots, f_n+a_ng)\in U_n(R)$.

 \item [(3)] The {\it Bass stable rank} of $R$, denoted by $\bsr R$,  is the smallest integer $n$ such that every element in $U_{n+1}(R)$ is reducible. 
 If no such $n$ exists, then $\bsr R=\infty$. 
  \end{enumerate}
  \end{definition}
  
  Note that if  $\bsr R=n$, $n<\infty$, and $m\geq n$,
 then every invertible $(m+1)$-tuple $(\bs f,g)\in R^{m+1}$ is reducible \cite[Theorem 1]{va}.\\
 
   Let $\N=\{0,1,2,\dots\}$. For the sequel, we make the  convention that the symbol 
 $\subset$  denotes {\it strict} inclusion.
 
 \begin{definition}
Let $R$ be a commutative unital ring, $R\not=\{0\}$.  
\begin{enumerate}
\item [(1)] A chain $\mathfrak C=\{I_0, I_1,\dots,  I_n\}$ 
of ideals in $R$  is  said to have {\it length $n$} ($n\in \N$), if 
$$I_0\subset  I_1\subset \dots\subset I_n,$$
the inclusions being strict.  We also call $\mathfrak C$ an  {\it $n$-chain}.
 Note the length of a chain $\mathfrak C$ counts
  the number of strict inclusions between the  members of $\mathfrak C$ and not the cardinal of
  $\mathfrak C$.
\item [(2)]   The {\it Krull dimension},  ${\rm Krd}\, R$,  of $R$ is defined to be the
supremum of the lengths of all increasing chains 
$$P_0\subset \cdots \subset P_n$$
of  {\it prime} ideals in $R$. Note that ${\rm Krd}\, R\in \{0,1,\dots,\infty\}$.

\end{enumerate}
\end{definition}

  Here is now Vasershtein's theorem (and  its  proof sketched in \cite{va}).
     
    \begin{theorem} [Vasershtein] \label{mainbsrpoly}
$$ \bsr \R[x_1,\dots,x_n]=n+1.$$ 
  \end{theorem}
  
  \begin{proof}  
  We first show that  $\bsr \R[x_1,\dots,x_n]\geq n+1$.
    Consider the  invertible $(n+1)$-tuple
 $$(x_1,\dots, x_n, 1-\sum_{j=1}^n x_j^2)$$
 in $ \R[x_1,\dots,x_n]$.
This tuple cannot be reducible in $\R[x_1,\dots,x_n]\ss C(\R^n,\R)$, since otherwise
the $n$-tuple $(x_1,\dots, x_n)$, restricted to the unit sphere $\partial\B$ in $\R^n$, would have a
zero-free extension $\bs e$ to the unit ball $\B$, where   $\bs e$ is given by 
$$\Bigl(x_1+u_1\cdot (1-\sum_{j=1}^n x_j^2),\;\dots\;, x_n+u_n \cdot(1-\sum_{j=1}^n x_j^2)\Bigr)$$
for some $u_j\in \R[x_1,\dots,x_n]$. This means  that $\bs e$
does not take the value $(0,\dots, 0)$ on $\B$. This contradicts   Brouwer's result
 that  the identity map 
$$(x_1,\dots,x_n): \partial \B\to \partial\B$$
defined on the  boundary  of   the closed unit ball  $\B$ in $\R^n$
does not admit a zero-free  continuous   extension  to $\B$.

Next we prove that 
$\bsr  \R[x_1,\dots,x_n]\leq n+1$.
This follows from a combination of  Theorem \ref{bass} below, telling us that the Bass stable
rank of a Noetherian ring with Krull dimension $n$ is less than or equal to $n+1$,
and Theorem \ref{krudim},  according to which the Krull dimension of $\R[x_1,\dots, x_n]$ is
$n$.
\end{proof}

In the next section we shall now present  analytic proofs of both Theorems mentioned above.
They were given by Estes and Ohm  for Theorem \ref{bass} (\cite{eo})  and Coquand and Lombardi
for Theorem \ref{krudim} (\cite{colo}). 
 
\begin{remark}
Concerning the polynomial ring $\C[z_1,\dots,z_n]$,
 to the best of our knowledge, the exact value of the Bass stable rank for 
  $\C[z_1,\dots,z_n]$ is not yet known. Only estimates are available:
$$\bsr \C[z_1,\dots,z_n]\leq n+1,$$
(follows as in the proof for $\R[x_1,\dots,x_n]$  because the Krull dimension of
$\C[z_1,\dots,z_n]$ is also $n$ by Theorem \ref{krudim}) and 
$$\bsr \C[z_1,\dots,z_n] \geq \left\lfloor \frac{n}{2}\right\rfloor +1$$
(see \cite{ga}, \cite{gg}).
\end{remark}

  %%%%%%
  %%%%%%
  \section{The Krull dimension of $\R[x_1,\dots,x_n]$}
  
 We begin with  our own proof of  the Coquand-Lombardi result  
  concerning an elementary  characterization of the Krull dimension of a commutative unital ring.
  We avoid the algebraic tool of considering localized rings and explicitely construct
  (with the help of Zorn's Lemma) chains of prime ideals having the correct length.
  Our tool will be the following standard result, which we would like to present, too.

 \begin{lemma}[Krull]\label{krulllemma}
 Let $R$ be a commutative unital ring, $R\not=\{0\}$, and 
 $S$  a multiplicatively closed set in $R$ with $1\in S$. Suppose that
 $I$ is an ideal in $R$ with $I\inter S=\emp$. Then there exists a prime ideal
 $P$ with $I\ss P$ such that $P\inter S=\emp$.
 \end{lemma}
 \begin{proof}
 Let $\mathscr V$ be the set of all ideals $J$ with $I\ss J$ and $J\inter S=\emp$.
 Then $\mathscr V\not=\emp$ because $I\in \mathscr V$, and $\mathscr V$  is partially ordered by set inclusion. If $\mathfrak C$ is any increasing chain
 in $\mathscr V$  then $\Union_{J\in \mathfrak C} J$ obviously is an ideal belonging to $\mathscr V$.
 Hence, by Zorn's Lemma, $\mathscr V$   admits a maximal element $P$.  Since $1\in S$
 and $S\inter P=\emp$, we obtain that $P\subset R$, the inclusion being strict.  Moreover, $I\ss P$.
 We claim that $P$ is a prime ideal.  For $f,g\in R$, let $fg\in P$ and suppose that neither $f$ nor $g$ belongs to $P$.
 Since $P$ is maximal in $\mathscr V$,  the ideals $P+fR$ and $P+gR$ meet $S$.
 Hence there exists $s,s'\in S$, $p,p'\in P$, and $r,r'\in R$ such that
 $$\text{$s=p+rf$ and $s'=p'+r'g$.}$$
 Multiplying both terms, we obtain
 $$ss'=pp'+ (r'g)p +(rf)p' + rr'(fg)\in P.$$
 Since $S$ is multiplicatively closed, $ss'\in  S$.  Hence $P\inter S\not=\emp$;
 a contradiction.  We conclude that $P$ is prime.
\end{proof}

  For $a,x\in R$ and $n\in \N$, let $L_{a,n,x}(y)=a^n(y+ax)$. If $a_j,x_j\in R$
and $n_j\in \N$ are given,
then we abbreviate $L_j(y):=L_{a_j,n_j,x_j}(y)$. 

\begin{lemma}\label{vor}
Let $Q_0$ and $Q_1$ be  ideals in a commutative unital ring $R$ with $Q_0\subset Q_1$ and let 
$a\in Q_1\setminus Q_0$. Suppose that  $Q_0$ is prime and that 
for $r,x\in R$ and $n\in \N$ we have 
$$a^n(r+ax)\in Q_0.$$
Then  $r\in Q_1$.
\end{lemma}
\begin{proof}
Because $a\notin Q_0$, the primeness of $Q_0$ implies that $r+ax\in Q_0\ss Q_1$.
Since $a\in Q_1$, we conclude that $r\in Q_1$. 
\end{proof}

\begin{theorem}[Coquand-Lombardi]\label{clmain}
Let $R$ be a commutative unital ring, $R\not=\{0\}$.   For $N\in \N$, 
the following  assertions are equivalent:
\begin{enumerate}
\item [(1)] The Krull dimension of $R$ is at most $N$.
\item [(2)] For all $(a_0,\dots, a_{N})\in R^{N+1}$ there exists  
$(x_0,\dots, x_{N})\in R^{N+1}$ and 
$(n_0,\dots, n_{N})\in \N^{N+1}$ such that
\begin{equation}\label{rekursi}
 a_0^{n_0}\Biggl( a_1^{n_1} \Bigl(\cdots \bigl( a_{N}^{n_{N}}(1 +a_{N} x_{N})+\cdots\bigr)\Bigr) +a_0x_0\Biggr)=0,
\end{equation}
in other words 
\begin{equation}\label{krullgen}
L_0\circ\dots\circ L_N(1)=0.
\end{equation}
\end{enumerate}
\end{theorem}

\begin{proof}
We show the contraposition of the assertion.\\

$\neg(1) \imp \neg (2)  $ ~~  Assume that the Krull dimension  of $R$ is at least $N+1$. Then 
$R$ admits a strictly increasing  ($N+1$)-chain
$$P_0\subset \dots  \subset P_N \subset P_{N+1}$$
of prime ideals $P_j$.  Choose $a_j\in P_{j+1}\setminus P_{j}$,  $j=0,\dots, N$.
If we suppose, contrariwise, that  \zit{krullgen} holds,
then $$a_0^{n_0}\Bigl(L_1\circ \dots\circ L_N(1)  +a_0x_0\Bigr)=0 \in P_0.$$
Hence, by Lemma \ref{vor}, $L_1\circ \dots\circ L_N(1)\in P_1$.  Now
$$L_1\circ \dots\circ L_N(1)= a_1^{n_1}\Bigl(L_2\circ \dots\circ L_N(1)  +a_1x_1\Bigr)\in P_1.$$
Hence, by Lemma \ref{vor},  $L_2\circ \dots\circ L_N(1)\in P_2$.
Continuing in this way, we deduce that 
$$a_{N}^{n_{N}}(1+a_{N}x_{N})= L_N(1)\in  P_{N}. $$
Hence, by Lemma \ref{vor}, $1\in P_{N+1}$; a contradiction.\\

$\neg(2) \imp \neg (1)  $ ~~   Suppose that there exists $\bs a=(a_0,\dots, a_N)\in R^{N+1}$
 such that for all
$\bs x\in R^{N+1}$ and $\bs n\in \N^{N+1}$ one has
 $$L_0\circ\dots \circ L_N(1)\not=0.$$
For $j=0,\dots, N$,  let 
 $$S_j=\Bigl\{L_j\circ \dots\circ L_N (1): \bs x\in R^{N-j+1}, \bs n\in \N^{N-j+1}\Bigr\}.$$
 Then $S_N\ss S_{N-1}\ss\dots\ss S_0$, or, what is the same, $S_0^c\ss S_{1}^c\ss\dots\ss S_N^c$
 where $S^c$ denotes the the complement of $S$.
 Note that $\{1, a_j,\dots, a_N\}\ss S_j$.  We claim that $S_j$ is multiplicatively closed.
 This follows by an inductive argument on $N-j$.  If $j=0$, then 
 $$S_N=\{a_N^{n_N}(1+a_Nx_N): x_N\in R, n_N\in \N\}$$
 is easily seen to be multiplicatively closed.  If for some $j$, $S_{N-j}$ is multiplicatively closed,
  then we use that
 $$L_{N-(j+1)}\circ L_{N-j}\circ \dots \circ L_N(1)= a_{\ssc N-(j+1)}^{\ssc n_{N-(j+1)}}
 \Bigl( L_{N-j}\circ \dots\circ L_N(1) + a_{\ssc N-(j+1)}x_{\ssc N-(j+1)}\Bigr)$$
 and observe that 
 $$a^n(s+ax)\cdot a^m(s'+a x')= a^{n+m}(ss' +a x'')$$
 where $x''= sx'+xs'+axx'$.

 {\bf Step1~~ } Looking at the zero ideal $I:=\{0\}$, and noticing that by assumption $0 \notin S_0$,
 Krull's Lemma \ref{krulllemma} tells us that  there
 exists a prime ideal
 $P_0$ with $P_0\inter S_0=\emp$, or in other words, $P_0\ss S_0^c$. 
 Now $a_0\in S_0$ implies that $a_0\notin P_0$. We claim that
\begin{equation}\label{erstes}
P_0\subset I_R(P_0, a_0) \ss S_1^c.
\end{equation}
 In fact, if the second inclusion does not hold,  then there is $p_0\in P_0$ and $x\in R$ such that
 $$p_0+xa_0=L_1\circ \dots\circ L_{N}(1) \in S_{1}\ss S_0.$$
 Hence
 \begin{eqnarray*}
p_0 &=&L_1\circ \dots\circ L_{N}(1) -xa_0\\
&=& a_0^0\Bigl(  L_1\circ \dots\circ L_N(1) -x a_0\Bigr)\in S_0.
  \end{eqnarray*}
Thus $p_0\in P_0\inter S_0$; a contradiction to the choice of $P_0$. Hence the inclusions 
\zit{erstes} hold.\\

{\bf Step 2~~} Now we apply Krull's Lemma again to get a prime ideal $P_1$ with
$$P_0 \subset  I_R(P_0, a_0)\ss  P_1\ss S_{1}^c.$$
Observe that $a_1\notin P_1$ because $a_1\in S_1$. We claim that
 \begin{equation}\label{zweite}
P_1\subset I_R(P_1, a_1) \ss S_2^c.
\end{equation}
 In fact, if the second inclusion does not hold, then there is $p_1\in P_1$ and $x\in R$ such that
 $$p_1+xa_1=L_2\circ \dots\circ L_{N}(1) \in S_{2}\ss S_1.$$
 Hence
 \begin{eqnarray*}
p_1 &=&L_2\circ \dots\circ L_{N}(1) -xa_1\\
&=& a_1^0\Bigl(  L_2\circ \dots\circ L_N(1) -x a_1\Bigr)\in S_1.
  \end{eqnarray*}
Thus $p_1\in P_1\inter S_1$; a contradiction to the choice of $P_1$. Hence the inclusions 
\zit{zweite} hold.\\

{\bf Step N~~} 
Continuing in this way,  we get a chain of prime ideals $P_j$, $j=0,\dots,N$, with 
$$P_0\subset P_1\subset\dots\subset P_N\ss S_N^c.$$
Observe that $a_N\notin P_N$ because $a_N\in S_N$. Hence $P_N$ is a proper ideal.
Therefore,  $P_N$ is contained in  a maximal ideal $M$. 
We claim that $M\inter S_N\not=\emp$. 
%%%
 In fact, if $a_N\in M$, then we are done. If $a_N\notin M$, then $I_R(a_N,M)=R$. In other words, 
 there is $x\in R$ and $m\in M$ such that $-a_Nx+m=1$; that is $1+a_Nx\in M$.
By the definition of  $S_N$, $1+a_Nx\in S_N$.  Hence $1+a_Nx\in M\inter S_N$. 
Thus $M\inter S_N\not=\emp$.\\

Hence $P_N\subset M$, the inclusion being strict,  and so  we have found a chain of prime ideals
of length $N+1$:
$$P_0\subset P_1\subset\dots\subset P_N\subset M.$$
We conclude that the Krull dimension of $R$ is at least $N+1$.
\end{proof}

The proof that the Krull dimension of $\R[x_1,\dots,x_n]$ is $n$ now works as in Coquand and Lombardi's paper:

\begin{proposition} [Coquand-Lombardi]\label{algindep}
Let $F$ be  a field and $R\not=\{0\}$ a commutative unital algebra over $F$. If any $(n+1)$-tupel 
$(f_0,\dots,f_n)\in R^{n+1}$ is algebraically dependent over $F$, 
that is, if  there is a non-zero polynomial $Q\in F[y_0,\dots, y_n]$ such that
 $Q(f_0,\dots, f_n)=0$,
 then the Krull dimension of $R$ is at most $n$.
\end{proposition}

\begin{proof}
Let $Q(f_0,\dots,f_n)=0$ for some  non-zero polynomial $Q\in F[y_0,\dots, y_n]$.
   We assume that the monomials
are ordered lexicographically with respect to the powers $(i_0, i_1,\dots, i_n)\in \N^{n+1}$.
This means that $(i_0,i_1,\dots, i_n)\preceq (j_0, j_1,\dots, j_n)$
if either $i_0<j_0$ or if there is $m\in \{0,\dots, n-1\}$ such that $i_\nu=j_\nu$ for all $\nu$ with 
$0\leq \nu\leq m$ and $i_{m+1}< j_{m+1}$.
Let $$a_{i_0,\dots,i_n} f_0^{i_0} f_1^{i_1}\dots f_n^{i_n}$$ be the "first" monomial 
appearing in the relation above (here the coefficient $a_{i_0,\dots,i_n}$ belongs to $F$
and $(i_0,\dots, i_n)\in\N^{n+1}$). Without loss of generality we may assume that the
coefficient of this monomial is 1. 
Then  $Q(f_0,\dots,f_n)$ can be written as
\begin{eqnarray*}
Q=f_0^{i_0}\dots f_{n-1}^{i_{n-1}}f_n^{i_n} &+&  f_0^{i_0}\dots 
f_{n-1}^{i_{n-1}}f_n^{1+i_n}R_n + f_0^{i_0}\dots f_{n-1}^{1+i_{n-1}}R_{n-1}+\dots\\
&+&  f_0^{i_0} f_1^{1+i_1} R_1 + f_0^{1+i_0} R_0
\end{eqnarray*}
where $R_j$ belongs to $F[f_j, f_{j+1},\dots, f_n]$, $j=0,1,\dots, n$. 
Hence $Q$ has been written in  the form given by equation \ref{rekursi}
(with $a_j:=f_j$ and $x_j:=R_j$), that is
$$f_0^{i_0}\Biggl( f_1^{i_1} \Bigl(\cdots \bigl( f_{n}^{i_{n}}(1 +f_{n} R_{n})+\cdots\bigr)\Bigr) +f_0R_0\Biggr)=0.
$$
We conclude from Theorem \ref{clmain}, that the Krull dimension of  $R$ is at most $n$.
\end{proof}

In order to show that $\R[x_1,\dots, x_n]$ satisfies the assumptions of  Proposition \ref{algindep}
and to deduce its Krull dimension, we need some additional information.

\begin{proposition}\label{cooprime}
Let $R=F[x_1,\dots,x_n]$. Then the ideals 
$I_R(x_1)$, $I_R(x_1,x_2)$, \dots, $I_R(x_1,\dots, x_j)$ with $1\leq j\leq n$ are prime ideals.
\end{proposition}
\begin{proof}
We may assume that $1\leq j<n$, since the ideals $I_R(x_1,\dots,x_n)$ are
 maximal, hence prime.
First we observe that  $F[x_1,\dots, x_n]=F[\bs x] [x_1,\dots, x_j]$, the polynomial ring
 with indeterminates $x_1,\dots, x_j$ and coefficients from the ring $F[\bs x]$,
where $\bs x=(x_{j+1},\dots, x_n)$.  Then every $f\in R$ can uniquely be written as
$$f(x_1,\dots, x_n)= \sum a_{\ell_{1},\dots\ell_{j}} x_1^{\ell_1}\dots x_j^{\ell_j},$$
where $a_{\ell_{1},\dots, \ell_{j}}\in F[\bs x]$.

Let us now consider the surjective ring-homomorphism
$$h: \begin{cases} F[x_1,\dots,x_n] &\to  F[\bs x]  \\
f(x_1,\dots,x_n) &\mapsto a_{\underbrace{\sc 0,\dots, \sc 0}_{j}}(x_{j+1},\dots, x_n),
\end{cases}
$$
where $a_{\sc 0,\dots, \sc 0}(x_{j+1},\dots, x_n)$ is the coefficient of
the monomial $x_1^0\dots x_j^0$. Then the kernel of $h$ is the union of the zero-polynomial
with the set   of
polynomials in $F[x_1,\dots, x_n]$ 
all of whose  summands  contain at least one of the indeterminates
 $x_1,\dots,x_j$. \footnote{ For example in the case $j=2$ and $n=3$,  $x_1x_2\mapsto 0$, $x_1x_2x_3\mapsto 0 $   and $x_2+x_3\mapsto x_3$.} 
  
 We conclude that the kernel of $h$ coincides with the ideal $I:=I_R(x_1,\dots, x_j)$.
 Hence $R/I$ is isomorphic to $F[\bs x]$. Since $F[\bs x]$ is an integral domain,
 we conclude  that $I$ is a prime ideal.
\end{proof}

 \begin{proposition}\label{dimvector}
  Let $F$ be a field. The dimension of the $F$-vector space $V_m(x_1,\dots, x_n)$
   of all polynomials  $p\in F[x_1,\dots, x_n]$ with $\deg p\leq m$ is $\binom {n+m}{n}$.
\end{proposition}
\begin{proof}
We proceed by induction on $n$. If $n=1$, then 
$$V_m(x)=\left\{\sum_{j=0}^m a_j x^j: a_j\in F\right\}$$
has dimension $m+1$, which coincides with $\binom {1+m}{1}$.
Now suppose that the formula holds for all  $\nu$ with $1\leq \nu\leq n$.  
If $x_1^{j_1}\dots x_{n}^{j_n} x_{n+1}^{j_{n+1}}$ is a monomial with 
$\sum_{i=1}^{n+1} j_i\leq m$, then for fixed $j:=j_{n+1}\in \{0,1,\dots, m\}$
we necessarily must have  $\sum_{i=1}^n j_i\leq m-j$.
Hence, by induction hypothesis, we have $\binom{n+m-j}{n}$
possibilities to choose these exponents $j_1,\dots, j_n$.
On the whole, we have
$$L:=\sum_{j=0}^m  \binom{n+m-j}{n}=\binom{n}{n}+\binom{n+1}{n}+\dots+ \binom{n+m}{n}
$$ 
choices.  But,
 $L= \binom{n+1+m}{n+1}$. Thus we are done.
\end{proof}

\begin{remark}\label{abz}
We also obtain that there are exactly  $\binom {n+m}{n}$ tuples  $(j_1,\dots,j_n)\in \N^n$
with $\sum_{j=1}^n j_i\leq m$.
\end{remark}

The following result is due to Perron \cite{pe}. We present a proof given to us by
Witold Jarnicki. 

\begin{theorem}[Perron]\label{perron}
Let $p_1,\dots, p_{n+1}$ be polynomials in $F[x_1,\dots,x_n]$.  Then there exists
a non-zero polynomial $P\in F[y_1,\dots,y_{n+1}]$ in $n+1$ variables such that
$$P(p_1,\dots,p_{n+1})=0.$$
\end{theorem}
\begin{proof}
We may assume that  the  polynomials $p_j$ are different from 0 (otherwise take 
$P(y_1,\dots,y_{n+1})= y_{n_0}$, where $p_{n_0}\equiv 0$.)
Let $k:=1+\max_{1\leq j \leq n+1} \deg p_j$. For big $L\in \N$, to be determined later,  
we are looking for 
$P\in F[y_1,\dots,y_{n+1}]$ with $0\leq \deg P\leq L$ and $P(p_1,\dots,p_{n+1})=0.$\\
Let  $V$ be  the vector space of all polynomials $p$ in $F[x_1,\dots,x_n]$ with $\deg p\leq kL$.
Then, by Proposition \ref{dimvector}, 
$$\dim V= \binom{kL+n}{n}=:A(L).$$
Consider now the  following collection $\mathcal C$  of polynomials:
$$ p_1^{j_1}\dots p_{n+1}^{j_{n+1}}:\; j_i\in \N, \;\sum_{i=1}^{n+1} j_i\leq L.$$
 Note that at this point we do not yet consider the {\bf set} of these  polynomials, because 
they may not be pairwise distinct.

 Each member of $\mathcal C$ 
   belongs to $V$,  because  for $p\in \mathcal C$,
 $$\deg  p\leq k(j_1+\dots +j_{n+1})\leq kL.$$
 If two members of $\mathcal C$ coincide, say 
 $$ p_1^{j_1}\dots p_{n+1}^{j_{n+1}}= p_1^{j_1^*}\dots p_{n+1}^{j_{n+1}^*},$$
where $(j_1,\dots, j_{n+1})\not= (j_1^*,\dots, j_{n+1}^*)$,  then  we let
 $$F(y_1,\dots,y_{n+1})=  y_1^{j_1}\dots y_{n+1}^{j_{n+1}}- 
 y_1^{j_1^*}\dots y_{n+1}^{j_{n+1}^*}$$
 and we are done. So let us assume that all the members of  $\mathcal C$ are distinct.
Let $S$ be the  set of all these members from $\mathcal C$. 
Then, by Remark \ref{abz}, 
 $${\rm card}\, S= \binom{L+n+1}{n+1}=: B(L).$$
Recall that $S\ss V$. 
We claim that $B(L)> A(L)$ for some $L$ (depending on $n$). \\
In fact, looking upon $B(L)$ and $A(L)$ as polynomials in $L$, we have that
$\deg B=n+1$  and $\deg A=n$. Thus, for large $L$, we obtain that $B(L)>A(L)$.\\
Thus the cardinal of set $S$ is strictly bigger than the dimension of the vector space $V$
it belongs to. Hence $S$ is a linear dependent set in $V$.  In other words,
there is a non-trivial linear combination of the elements from $S$ that is identically zero.
This implies that there is a non-zero polynomial $P\in F[y_1,\dots, y_{n+1}]$ of degree  at most $L$
such that $P(p_1,\dots, p_{n+1})=0$.
\end{proof}

%\newpage

\begin{theorem}\label{krudim}
If $F$ is a field then the Krull dimension of $F[x_1,\dots,x_n]$ is  $n$.
\end{theorem}
\begin{proof}
By Perron's Theorem \ref{perron}, $R:=F[x_1,\dots,x_n]$ satisfies the assumption of 
Proposition \ref{algindep}. Hence the Krull dimension of $R$ is less than or equal to $n$.
By Proposition \ref{cooprime}, we have a chain of prime ideals
$$\{0\}\subset I_R(x_1)\subset  I_R(x_1,x_2)\subset \dots\subset I_R(x_1,\dots, x_n).$$
Since $\{0\}$ is a prime ideal too, this chain has length $n$.
Thus the Krull dimension of $R$ is $n$.
\end{proof}

  \section{Anderson's approach to Noether's  minimal prime theorem}
  
  To prove Bass' Theorem along the lines developed by Estes and Ohm \cite{eo},  we need 
  to collect in  the  following  classical Theorem by E. Noether some information on the abundance of minimal prime ideals in Noetherian rings.
We present a recent 
proof developed  by D. Anderson \cite{and}.

\begin{theorem}\label{miniprime}
Let $R$ be a Noetherian ring. Then the system, $\mathscr P_{min}$, 
 of prime ideals containing a given proper ideal $I\ss R$
and that are minimal (with respect to set inclusion) is a non-empty finite set. 
\end{theorem}
\begin{proof}
Let $\mathscr P_I$ be the set of all prime ideals containing $I$. Then $\mathscr P_I\not=\emp$,
because there exists (using Zorn's Lemma) a maximal ideal $M$
 containing $I$.  Obviously $M$ is prime.  A second use of Zorn's Lemma   shows that 
 $\mathscr P_I$ also admits minimal elements. Hence $\mathscr P_{min}\not=\emp$. 
 \footnote{ For later purposes we note that, by the same reason,
  if $I$ and $P$ are ideals, $P$ prime and $I\ss P$, 
 then  there exist   minimal prime ideals  $P_{min}$ with 
 $I\ss P_{min}\ss P$.}\\ 

Next we consider the set $\mathfrak J$ of all ideals of the form
$$ P_1\odot \dots \odot P_n:=\Bigl\{\sum_{j=1}^m f_{1,j}\dots f_{n,j}:
 f_{k,j}\in P_k, \, m\in \N\bigr\}, $$
where $P_k\in \mathscr P_{min}$,  $n\in \N$. Since $R$ is a Noetherian ring,
every ideal in $\mathfrak J$ is finitely generated. \\

{\it Case 1}  If for some 
$J:= P_1\odot \dots \odot P_{n_0}\in \mathfrak J$ we have $J\ss I$,
then $ P_1 \cdots  P_{n_0}\ss I\ss P$ for every $P\in \mathscr P_{min}$. Hence, 
 the primeness of $P$ implies that there is $i_0\in \{1,\dots, n_0\}$
depending on $P$,  such that $P_{i_0}\ss P$  
%%%
(for if this is not the case, there exists for every $j\in \{1,\dots, n_0\}$ an element
$f_j \in P_j \setminus P$  with $f_1\cdots f_n \in I \ss  P$, 
contradicting the primeness of $P$).
Since $P$ is minimal, $P=P_{i_0}$.  Hence $\mathscr P_{min}=\{P_1,\dots, P_{n_0}\}$
and we are done.\\

{\it Case 2} 
Let us suppose that $J\not\ss  I$ for every $J\in \mathfrak J$. 
 The aim is to show that this case does not occur.
 Consider the set
$$\text{$\mathscr L:=\{L\ss R, \; L$ ideal,   $I\ss L$,  $J\not\ss L$ for each $J\in \mathfrak J\}$.}$$
Then $\mathscr L\not=\emp$ because, by assumption, $L:=I\in \mathscr L$. 
Moreover all ideals in $\mathscr L$ are proper, because $J\ss R$ for all $J\in \mathfrak J$.
With respect to set inclusion, $\mathscr L$ is partially ordered.  We claim that
\begin{equation}\label{maxiprim}
\mbox{$\mathscr L$ admits a maximal element $M$ and $M$ is automatically prime}.
\end{equation}
%%%
Suppose for the moment that this has been verified. Then,
using Zorn's Lemma,  there exists a minimal prime ideal $P$ over $I$
(hence $P\in \mathfrak J$)
with $I\ss P\ss M$. This is a contradiction, though, to the fact that $M\in \mathscr L$.
We conclude that this second case cannot occur.  
Hence, in view of the first case, $\mathscr P_{min}$ is finite.\\

Let us verify the two assertions in \zit{maxiprim}. To this end, 
let $\{L_\lambda:\lambda\in \Lambda\}$
be an increasing chain in $\mathscr L$. Let us show that
 $L:=\Union_\lambda L_\lambda\in \mathscr L$.
 
i)  $I\ss L$ is obviously satisfied and $L$ is an ideal because the chain is increasing.

ii) Let $J\in  \mathfrak J$. Note that $J$ is finitely generated, say $J=I_R(f_1,\dots, f_d)$, and that
$J\not\ss L_\lambda$ for any $\lambda$.
If we suppose (in view of  achieving a contradiction) that 
$J\ss  \Union_{\lambda\in \Lambda} L_\lambda=L$, then
$f_j\in L_{\lambda_j}$ for some $\lambda_j\in \Lambda$, $(j=1,\dots, d)$. 
The monotonicity of the chain implies that there exists $\lambda'$ such that $f_j\in L_{\lambda'}$
for all $j=1,\dots, d$. Thus $J\ss  L_{\lambda'}$, a contradiction to the hypothesis that 
$L_{\lambda'}\in \mathscr L$.

Thus we have shown that $\mathscr L$ is an inductive set and so, by Zorn's Lemma,
$\mathscr L$ admits a maximal element $M$.  In particular, $I\ss M$.  We claim that $M$ is prime.

To see this, we first observe that $M$ is proper since otherwise $J\ss M=R$ for 
every $J\in \mathfrak J$.  Now let $f,g\in R$ with $fg\in M$.  Suppose, to the contrary, 
that $f\notin M$ and $g\notin M$.
Since $M$ is a maximal element in $\mathscr L$, and $I\ss M$, there is $J_f\in \mathfrak J$
and $J_g\in \mathfrak J$ such that 
$$\text{$ J_f\ss I_R[f, M]$ and $J_g\ss I_R(g,M)$}.$$
By the definition of $\mathfrak J$, there exists  $P_i$ and $P_k$ in $ \mathscr P_{min}$
with $I\ss P_i\ss J_f$ and $I\ss P_k\ss J_g$.  We claim that  
$$P_i\cdot P_k\ss I_R(fg,M)\ss M.$$
In fact, if $p_i\in P_i$ and $p_k\in P_k$, then there are $x_i,x_k\in R$ and $m_i,m_k\in M$ such that
$$p_ip_k=(x_i f+m_i)(x_k g+m_k)= x_ix_k(fg)+ m_i(x_kg) +m_k(x_if) +m_im_k\in M.$$
Because $M$ is an ideal, $P_i\odot P_k\ss M$.  Thus we have found an element in $\mathfrak J$
that is contained in $M$. Since $M\in \mathscr L$, this is a contradiction.
This proves that $M$ is prime.
\end{proof}

%%%%%%%%
  
  \section{The Estes-Ohm approach}
  
  Here we present the approach  to Bass' Theorem given by Estes and Ohm \cite{eo}.
  
  \begin{lemma}\label{unionof2}
Let $R$ be a commutative ring and $I_0,I_1,I_2$  three ideals with $I_0\ss I_1\union I_2$.
Then  $I_0\ss I_1$ or $I_0\ss I_2$.
\end{lemma}

\begin{proof}
Suppose that neither $I_0\ss I_1$ nor $I_0\ss I_2$.
Then there are $a_1\in I_0\setminus I_2\ss I_1$ and
 $a_2\in I_0\setminus I_1\ss I_2$. Since $I_0$ is an ideal, $s:=a_1+a_2\in I_0\ss I_1\union I_2$.
 Without loss of generality we may assume that $s\in I_1$. Then
 $ a_2=s-a_1\in I_1$; a contradiction.  We conclude that  $I_0\ss I_1$ or $I_0\ss I_2$.
\end{proof}

\begin{remark}
The assertion above does not hold  (in general) for unions of three (or more) ideals. In fact,
let $R$ be  a finite ring such that not all of its maximal ideals are principal. Let $I\ss R$ be  
 a non-principal maximal ideal. Then  
$N={\rm card}\; I\geq 4$ and  $I=\Union_{j=1}^N R x_j$. 
But of course, $I_j:=Rx_j $ does not contain $I$. A specific example is, for instance, the 
quotient ring $R=\Z_2[x,y]/ M$ where $M$ is the ideal generated by $x^2$, $xy$ and $y^2$.
When denoting the equivalence class   of $u\in \Z_2[x,y]$ by $\tilde u$, we have
$$R=\{ \tilde 0, \tilde 1, \tilde x,\tilde y, \tilde x+\tilde y, \tilde 1+\tilde x+\tilde y \},$$
and as $I$ we may take $I=I_R(\tilde x,\tilde y)$, which coincides with the set
$\{\tilde 0, \tilde x,\tilde y, \tilde x+\tilde y\}$.
\end{remark}

If one stipulates,  however,  that the ideals $I_j$ are prime, then one obtains the following
well-known result. For the readers' convenience  we present its proof, too.
 
  \begin{lemma}\label{unionofideals}
Let $R$ be a commutative ring and $I, P_1,\dots, P_m$  ideals in $R$ such that each
$P_j$ is prime and   
$$I\ss\Union_{j=1}^m P_j.$$
 Then $I\ss P_{j_0}$ for some $j_0\in \{1,\dots,m\}$. 

\end{lemma}
\begin{proof}
If $m=1$, then nothing has to be shown.
If $m=2$, then the assertion holds by Lemma \ref{unionof2}.   So we may assume that $m\geq 3$. 
For $j\in \{1,\dots,m\}$, let 
$$Q_j=\Union_{k\not=j} P_k.$$
We claim that  there is $j$ such that $I\ss Q_j$.
Suppose, to the contrary,  that  for every $j$, $I\not\ss Q_j$. 
Then we may choose for $j\in \{1,\dots,m\}$ elements  $a_j\in I\setminus Q_j$.
Note that $I\setminus Q_j\ss P_j$.
Let
$$b_j=\prod_{k\not=j} a_k.$$
Then $b_j\in I$ and $b_1+\dots+ b_m\in I\ss \Union_{k=1}^m P_k$.
Hence there is $j\in \{1,\dots,m\}$ such that $b_1+\dots+ b_m\in P_j$.
Since  $b_k\in P_j$ for every $k\not=j$, we deduce  that 
$\sum_{k\not=j} b_k \in P_j$. Therefore $b_j\in P_j$.
Since $P_j$ is prime,  there is $k\not=j$ such that $a_k\in P_j\ss Q_k$. 
This is a contradiction to the choice of the elements $a_1,\dots, a_m$. We conclude that
$I\ss \Union_{k\not=j} P_k$ for some $j$. 

Now we proceed by backwards induction to reduce the number of  prime ideals
up to the case $m=2$. That case, though, is handled by Lemma  \ref{unionof2}.\\

\end{proof}

\begin{lemma}\label{eo1}
Let $P_j$ be prime ideals in a commutative unital ring $R$ with $P_j\not\ss P_k$ for $j\not=k$,
$j,k\in\{1,\dots, m\}$. Let $a, r\in R$. Then there exists $b\in R$ such that for all $j\in \{1,\dots,m\}$:
\begin{equation} \label{eorel}
\text {if  $a\notin P_j$ then $r+a\,b\notin P_j$}.
\end{equation}

\end{lemma}
\begin{proof}
If $r\in\Inter_{j=1}^m P_j$, then we may choose  for $b$ any element not in  $\Union_{j=1}^m P_j$.
If $r\notin \Union_{j=1}^m P_j$, then we take $b_j\in P_j$ and  let $b=\prod_{j=1}^m b_j$.
In the remaining cases, modulo a re-enumeration, we may  assume that $r\in\Inter_{j=1}^s P_j$ but 
$r\notin \Union_{j=s+1}^m P_j$. By Lemma \ref{unionofideals}, the hypothesis $P_j\not\ss P_k$ for $j\not=k$
implies that  $P_j\not\ss \Union_{k=1}^s P_k$ for each $j\in \{s+1,\dots, m\}$. Let
$b_j\in P_j\setminus  \Union_{k=1}^s P_k$, $j\geq s+1$, and let
$$b= b_{s+1}\cdot\, \cdots\,\cdot  b_{m}.$$
Then $b\in \Inter_{j=s+1}^m  P_j$. But $b\notin P_k$ for $1\leq k\leq s$, because otherwise
the primeness of $P_k$ implies that one of the factors $b_j$ with $s+1\leq j\leq m$ belongs to $P_k$,
a contradiction to the choice of $b_j$.

Fix $j_0\in \{1,\dots,m \}$ and assume  that $a\notin P_{j_0}$.
We claim that $r+ab\notin P_{j_0}$. In fact, assuming the contrary, let  $u:=r+ab\in P_{j_0}$.
If $1\leq j_0\leq s$, then  $r\in P_{j_0}$, hence $ab=u-r\in P_{j_0}$. The assumption on $a$ and 
the primeness of $P_{j_0}$ imply that $b\in P_{j_0}$, a contradiction. 
If $s+1\leq j_0\leq m$, then $b\in  P_{j_0}$. Hence $r=u-ab\in  P_{j_0}$;  this  is a contradiction
to the choice of $r$.
\end{proof}

  \begin{lemma}\label{eo2}
Let $R$ be a Noetherian ring.  Given $a, a_j\in R$, $j=0,1,\dots, s$, 
there exists $b_j\in R$, $j=1,\dots, s$,  such that for $i=1,2, \dots, s$, 
the (finite) set,  $\mathscr P_{i-1}$, of  minimal prime ideals $P$  containing 
$$I_R(a_0, a_1+b_1 a, \dots, a_{i-1}+ b_{i-1} a)\;  \footnote{ If $i=1$,
then we consider $I_R(a_0)$.} $$
has the following property:
\begin{equation}\label{notnot}
%\text{if $a_i'\in P$ then $a\in P$}.
\text{if $a\notin P$,  where $P\in \mathscr P_{i-1}$,  then $a_i+ b_i a\notin  P$}.
\end{equation}
\end{lemma}
\begin{proof}
To prove the assertions via induction, we will use several times Lemma \ref{eo1}.
 
\underline{$i=1$}:  Recall that $\mathscr P_0$ is the class of all minimal 
prime ideals $P$ with $I_R(a_0)\ss P$.
By Theorem \ref{miniprime},  $\mathscr P_0$ is finite; say $\mathscr P_0=\{P_1,\dots, P_t\}$.
To apply Lemma \ref{eo1}, we let $r=a_1$. This gives $b_1\in R$ such that for all 
$\nu\in \{1,\dots, t\}$
 $$\text{$a_1':=a_1+b_1a\notin P_\nu$ whenever $a\notin P_\nu$.}$$
Thus \zit{notnot} is satisfied.\\
Now let us suppose that $b_1,\dots, b_i$  have been constructed and that  \zit{notnot} is satisfied for some $i\in \{1,\dots, s-1\}$. Let $a_j':=a_j+ b_j a$, where $j\in \{1,\dots, i\}$.

\underline{$i\to i+1$}:    By definition,  $\mathscr P_i$ is the class of all minimal 
prime ideals $P$ with 
$$I_R(a_0,a_1',\dots, a_{i}')\ss P.$$
By Theorem \ref{miniprime},  $\mathscr P_i$ is finite. 
To apply Lemma \ref{eo1}, we let $r=a_{i+1}$. This gives $b_{i+1}\in R$ such that
for all $P\in \mathscr P_i$:
 $$ \text{$a'_{i+1}:=a_{i+1}+b_{i+1}a\notin  P$ whenever $a\notin  P$.}$$
Thus \zit{notnot} is satisfied for $i+1$.
\end{proof}

  \begin{theorem}[Bass] \label{bass}
Let $R$ be a Noetherian ring with Krull dimension less than or equal to $n$. 
Then $\bsr R\leq n+1$.
\end{theorem}
\begin{proof}
This is the proof given by Estes and Ohm \cite{eo}.
Let $\bs a:=(a_1,\dots,a_{n+1}, a)\in U_{n+2}(R)$. We have to show that $\bs a$ is reducible.
Choose $a_0=0$. Associate with $a_j$ the elements  $b_j$ comming from Lemma \ref{eo2},
$j=1,\dots, n+1$, and let $a_j'=a_j+b_ja$. For $i=1,\dots, n+1$, consider the ideals
$$I_i:=I_R (a_0, a_1',\dots, a_{i}'),$$
and let $\mathscr P_i$ be the (finite) set of minimal prime ideals $P$ with $I_i\ss P$.
 
Note that the reducibility of $\bs a$ is a consequence  to the assertion that $I_{n+1}=R$.  
Suppose, to the contrary, that $I_{n+1}$ is a proper ideal. 

  We claim that
\begin{equation}\label{a}
\mbox{$a\notin P$ for every $P\in \mathscr P_{n+1}$}
\end{equation}
In fact, if we suppose that $a\in P$ for some $P\in \mathscr P_{n+1}$,
then  the invertibility of $\bs a$ implies that 
$$P\supseteq I_R(a_1',\dots,a_{n+1}', a)=I_R(a_1,\dots, a_{n+1}, a)=R,$$
a contradiction to the fact that prime ideals are proper ideals.  Hence assertion \ref{a} holds.

Consider now the following chain of ideals:
$$I_R(a_0)\ss I_R(a_0, a_1')\ss \dots\ss I_R(a_0,a_1',\dots, a_{n}')\ss I_R(a_0,a_1',\dots,a_{n+1}')=I_{n+1}\subset R.$$

We shall construct a chain of prime ideals that has length $n+1$ 
(in other words $n+2$ elements) which
will yield a contradiction to the assumption that the Krull dimension of $R$ is 
less than or equal to $n$.

Let $P_{n+1}\in \mathscr P_{n+1}$. 
Choose a minimal prime ideal $P_n\in \mathscr P_n$
with 
$$I_R(a_0,a_1',\dots, a_{n}') \ss P_n \ss P_{n+1}.$$
Backwards induction yields  minimal prime ideals $P_i$ with
$$I_R(a_0,a_1',\dots, a'_{i})\ss P_i\ss P_{i+1}\ss\dots\ss P_{n+1}$$ 
($i=1,\dots,n$), and finally
a minimal prime ideal $P_{0}$ with 
$$I_R(a_0)\ss P_0\ss P_1.$$
We claim that all the inclusions in the chain
$$P_0\subset P_1\subset\dots\subset P_{n}\subset P_{n+1}$$
are strict. To do so, we use Lemma \ref{eo2}. Fix $i\in \{1,\dots, n+1\}$ and
consider ${I_{i-1}=I_R(a_0,a_1',\dots, a'_{i-1})}$,   with the convention that
$I_{0}=I_R(a_0)$.
We first observe that $a\notin P_{i-1}$,
since otherwise $a\in P_{n+1}$, a contradiction to \zit{a}.

Hence, by Lemma \ref{eo2}, $a_i'\notin P_{i-1}$. But by construction, $a'_i\in P_i$.
Thus $P_{i-1}\subset P_i$, the inclusion being strict.  Hence, under the assumption that
$I_{n+1}$ is proper, we have shown that the Krull dimension of $R$ is at least $n+1$.
This contradicts the hypothesis. Consequently, $I_{n+1}=R$.  Thus, as already mentioned,
$\bs a\in U_{n+2}(R)$ is reducible. Hence $\bsr R\leq n+1$.
\end{proof}

To conclude this section, let us mention the following generalization of Bass' Theorem
given by R. Heitmann  \cite{hei}. It shows that the Noetherian condition can be dropped.
\begin{theorem}
Let $R$ be a commutative unital ring  with Krull dimension $n$. If $R$ is
an integral domain,  then $\bsr R\leq n+1$.
If, on the other hand,  $R$ has zero-divisors, then $\bsr R \leq n+2$.
\end{theorem}

%%%%%%%%%%%%%%%

\section{The topological stable rank of $\R[x_1,\dots,x_n]$}

\begin{definition} 
Let $R$ be a ring endowed with  a topology $\mathcal T$ (we do not assume that the topology
is compatible with the algebraic operations $+$ and $\cdot$).
 The {\it topological stable rank}, ${\rm tsr}_{\mathcal T}R$, of $(R,\mathcal T)$ is the least integer
  $n$ for which $U_n(R)$ is dense in $R^n$, or infinite if no such $n$ exists.  
\end{definition}
If the ring  $R$ is endowed with two topologies $\mathcal T_1$ and $\mathcal T_2$ such that
$\mathcal T_1$ is weaker than $\mathcal T_2$, then 
$$\tsr_{\mathcal T_1} R \leq \tsr_{\mathcal T_2} R.$$

For the ring of polynomials, we work with the  topology of uniform convergence.

\begin{theorem}\label{tsrp}
The topological stable rank of $\R[x_1,\dots, x_n]$ is $n+1$.
\end{theorem}
\begin{proof}

We first prove that 
\begin{equation}\label{tsrmajr}
\tsr \R[x_1,\dots,x_n]\leq n+1.
\end{equation}
Let $\bs p:=(p_1,\dots, p_{n+1})$ be an $(n+1)$-tuple  in  $\R[x_1,\dots,x_n]$.
Now we look upon $\bs p$ as being an $(n+1)$-tuple in $\C[z_1,\dots,z_{n}]$.
Choose, according to Perron's Theorem \ref{perron}, a non-zero polynomial $P$  over $\C$ with
$n+1$ indeterminates such that  $$P(p_1,\dots, p_{n+1})=0.$$
Then $P$, looked upon as a polynomial function, vanishes identically on the 
image $\bs p(\C^n)$. 
But $P$ cannot vanish identically on the ball
$$B(\bs 0,\e):=\{\bs x=(x_1,\dots,x_{n+1})\in \R^{n+1}: ||\bs x||_2\leq \e\}, $$
since otherwise $P$ would be the zero-polynomial (just consider the partial derivatives
at the origin). Since $\e>0$ can be chosen arbitrarily
small, we obtain a null-sequence $(\bs\e_k)$  in $\R^{n+1}$ such that $\bs\e_k\notin \bs p(\C^n)$.
Hence the $(n+1)$-tuple $\bs p-\bs \e_k$
is invertible in $C(\C^n,\C)$.  From Theorem \ref{bezpolreal}
%By Hilbert's Nullstellensatz \ref{hil}, $\bs p-\bs e_k$ is an invertible
%&$(n+1)$-tuple in $\C[z_1,\dots,z_n]$, from  which 
we deduce that $\bs p-\bs \e_k$ is in $U_{n+1}(\R[x_1,\dots,x_{n}])$.
Since $\bs p-\bs \e_k$ uniformly approximates $\bs p$, we  are able  to conclude that 
$\tsr \R[x_1,\dots,x_n]\leq n+1$.\\

Next we show that $\tsr \R[x_1,\dots,x_n]\geq n+1$. \medskip
 
 Consider the identity map $\bs x=(x_1,\dots,x_n)$ of $\R^n$ onto $\R^n$. Note that 
 $$\bs x\in \underbrace{\R[x_1,\dots,x_n]\times \dots\times \R[x_1,\dots,x_n]}_{n-times}.$$
 Suppose that there exist invertible $n$-tuples in $\R[x_1,\dots,x_n]$ that uniformly approximate
 $\bs x$. That is, for every $\e>0$ there is $\bs f=(f_1,\dots,f_n)\in U_n(\R[x_1,\dots,x_n])$ such that
 $|x_j-f_j|<\e$ for $j=1,\dots, n$.  This implies of course that $x_j-f_j$ is a
 constant (we keep this generality though, since it  also works for non-polynomial rings).
 In particular, this inequality then holds on the unit sphere $S_{n-1}$. Let $A$ be  the
 Banach algebra $A=C(S_{n-1},\R)$ of all continuous, real-valued functions on $S_{n-1}$,
 endowed with the supremum norm.
 We obviously have that   $\bs f \in U_n(A)$.
 
  By a classical theorem in the theory of Banach algebras (\cite{pal} or \cite{ru})
  there exists a matrix $H\in M_n(A)$ with
 $$(x_1,\dots, x_n)^t=( \exp H)\, (f_1,\dots,f_n)^t$$
 whenever $\e$ is chosen small enough. Extending the entries of $H$ with the help of
 Tietze's Theorem  to continuous functions  on $\R^n$ we  obtain
 a zero-free extension of $\bs x|_{\partial \mathbf B_{n}}$ to $\mathbf B_{n}$.
As above (see the proof of Theorem \ref{mainbsrpoly}),  
this contradicts  Brouwer's fixed point theorem.
\end{proof}

\section{Acknowledgements}  

We thank
Leonhard Frerick for providing us the reference \cite{plo}; 
Witold Jarnicki for the proof of Theorem \ref{perron};
Amol Sasane for the reference \cite{and}
as well as Peter Pflug  and Thomas Schick for valuable comments yielding to a  proof 
of  $\tsr \R[x_1,\dots,x_n]\leq n+1$
  in Theorem \ref{tsrp}.

 \end{document}